\documentclass[amssymb,twoside,11pt]{article}
\usepackage{latexsym,amsmath,graphicx}
\usepackage{amsfonts}
\usepackage{enumitem}
\newtheorem{theorem}{Theorem}[section]
\newtheorem{lemma}[theorem]{{\bf Lemma}}

\newtheorem{rem}[theorem]{{\bf Remark}}
\newtheorem{ex}[theorem]{{\bf Example}}
\newtheorem{definition}{Definition}[section]
\numberwithin{equation}{section}
\newenvironment{proof}{\indent{\em Proof:}}{\quad \hfill
$\Box$\vspace*{2ex}}

\begin{document}
\setcounter{page}{1}
\begin{center}
\vspace{0.4cm} {\large{\bf On the Impulsive Implicit $\Psi$--Hilfer Fractional Differential Equations with Delay}}

\vspace{0.5cm}
Jyoti P. Kharade  $^{1}$\\
jyoti.thorwe@gmail.com\\
\vspace{0.35cm}
Kishor D. Kucche $^{2}$ \\
kdkucche@gmail.com \\

\vspace{0.35cm}
 $^{1,2}$ Department of Mathematics, Shivaji University, Kolhapur-416 004, Maharashtra, India.\\
 \end{center}
\def\baselinestretch{1.0}\small\normalsize
\begin{abstract}
In this paper, we investigate the  existence and uniqueness of solutions and derive the Ulam--Hyers--Mittag--Leffler stability results for impulsive implicit $\Psi$--Hilfer fractional differential equations with time delay. It is demonstrated that  the Ulam--Hyers and generalized Ulam--Hyers stability are the specific cases of Ulam--Hyers--Mittag--Leffler stability. Extended version of Gronwall inequality, abstract Gronwall lemma and Picard operator theory are the primary devices in our investigation. We give an example to illustrate the obtained results.

\end{abstract}

\def\baselinestretch{1.5}

\allowdisplaybreaks
\allowdisplaybreaks
\section{Introduction}
In the literature, a lot of consideration has been paid to analyze the impulsive fractional differential equations (FDEs)  in view of its applications in displaying  several real world phenomena that appearing in the applied sciences.
Diverse ideas of solutions of impulsive FDEs and the criterion to derive  existence and uniqueness results have been given in the review done by Wang et al.\cite{Fečkan}. Many inserting work on impulsive FDEs can be found in the literature that deals with existence, uniqueness, data dependence and stability of solutions, see for instance, the works of Wang et al. \cite{Wang1, Wang2,Wang22}, Feckan et al. \cite{fec},  Benchohra and Slimani \cite{Benchohra3}, Mophou \cite{Mophou} and the references given therein. 

On the other hand, Nieto et.al.\cite{JNieto} and Benchora et.al. \cite{Benchora2,Benchora3} have started the investigation of implicit FDEs and got intriguing outcomes  pertaining to existence and Ulam types stability. In 2016, Kucche et.al. \cite{Kucche} acquired  existence results along with data dependence of solutions  for implicit FDEs by means of fractional integral inequality and the $\epsilon$-approximated solutions. Shah et. al. \cite{Shah} investigated  existence and Hyers--Ulam stability of solution for implicit impulsive FDEs.

The existence and uniqueness of solutions and Ulam--Hyers--Mittag--Leffler (UHML) stability of different kinds of fractional differential and integral equations with time delay have been investigated in \cite{Eghbali,Zhang,Niazi} by using  Picard operator thoery and abstract Gronwall's lemma.

 Very recently, Liu et. al. \cite{Liu3}  considered $\Psi$-Hilfer FDEs and obtained existence, uniqueness and UHML stability of solutions  via Picard operator theory and a generalized Gronwall inequality involving  $\Psi-$Riemann--Liouville fractional integral.  In 2019, Sousa et al. \cite{KDK} analyzed impulsive FDEs involving $\Psi$-Hilfer  derivative.
 
 Contemplating the works referenced above, we firmly feel to consider the impulsive FDEs with generalized fractional derivative viz. $\Psi$-Hilfer fractional derivative that bring together several well known fractional derivatives. Such an investigation surely contribute to the fractional calculus. Propelled by this reality and motivated by the works of \cite{Shah}-\cite{Liu3}, in the present paper, we study the nonlinear implicit impulsive $\psi$-Hilfer fractional differential equation ($\psi$-HFDE) with time delay of the form:  
\begin{small}
\begin{align}
& ^\mathcal{H} \mathbb{D}^{\alpha,\, \beta; \, \Psi}_{0^+}u(t)=f\left(t, u(t), u(h(t)), ^\mathcal{H} \mathbb{D}^{\alpha,\, \beta; \, \Psi}_{0^+}u(t) \right) ,~t \in J=(0,b]-\{t_1, t_2,\cdots ,t_p\}, \label{p1}\\
&\Delta \mathbb{I}_{0^+}^{1-\rho; \, \Psi}u(t_k)= \mathcal{J}_k(u(t_k^-)),  ~k = 1,2,\cdots,p, \label{p2} \\
&\mathbb{I}_{0^+}^{1-\rho; \, \Psi}u(0)=u_0 \in \mathbb{R},  ~ \rho=\alpha+\beta-\alpha\beta, \label{p3} \\
& u(t)=\phi(t),\, t\in [-r,0], \label{p4} 
\end{align}
\end{small}
where $\Psi\in C^{1}(J,\mathbb{R})$ be an increasing function with $\Psi'(x)\neq 0$, for all $x\in J$, ~$^\mathcal{H} \mathbb{D}^{\alpha, \, \beta; \, \Psi}_{0^+}(\cdot)$ is the $\Psi$-Hilfer fractional derivative of order $\alpha~ (0<\alpha<1)$ and type $\beta ~(0\leq\beta\leq 1)$, ~ $\mathbb{I}_{0^+}^{1-\rho; \, \Psi}$ is left sided $\Psi$-Riemann Liouville fractional integration operator, $ ~0=t_0< t_1 < t_2 < \cdots < t_p < t_{p+1}=b $,~$
\Delta \mathbb{I}_{0^+}^{1-\rho; \, \Psi}u(t_k)= \mathbb{I}_{0^+}^{1-\rho; \, \Psi}u(t_k^+)- \mathbb{I}_{0^+}^{1-\rho; \, \Psi}u(t_k^-) 
$, 
~$
 \mathbb{I}_{0^+}^{1-\rho; \, \Psi}u(t_k^+) = \lim_{\epsilon\to 0^+} \mathbb{I}_{0^+}^{1-\rho; \, \Psi}u(t_k + \epsilon)$ 
~~\mbox{and}~~
$\mathbb{I}_{0^+}^{1-\rho; \, \Psi}u(t_k^-) = \lim_{\epsilon\to 0^-} \mathbb{I}_{0^+}^{1-\rho; \, \Psi}u(t_k + \epsilon)
$.  The functions $f:J\times \mathbb{R}^3 \to \mathbb{R}$ and $\mathcal{J}_k:\mathbb{R}\to \mathbb{R}$ are appropriate functions specified latter, $h\in C(J,[-r,b])$ is a continuous delay function such that $h(t)\leq t,~ t\in J$. Further, $\phi\in \mathcal{C}:=C([-r,0],\mathbb{R})-$ the space of all continuous functions from $[-r,0]$ to $\mathbb{R}$, with supremum norm $\|\phi\|_{\mathcal{C}}=\sup_{t\in [-r,0]}|\phi(t)|$. 

Our main objective is to investigate the existence and uniqueness of solutions and examine the UHML stability of impulsive implicit $\Psi$--HFDE \eqref{p1}--\eqref{p4} with time delay. It is observed that the Ulam--Hyers and generalized Ulam--Hyers stability for the problem \eqref{p1}--\eqref{p4} are obtained as particular cases of UHML stabilty results that we acquired. Our analysis is based on extended version of Gronwall inequality, abstract Gronwall lemma and the  Picard operator theory.

Results obtained in the present paper extends the works of \cite{Shah}-\cite{Liu3} and can be considered as a contribution to the developing field of fractional calculus with generalized fractional derivative operators. 

\section{Preliminaries} \label{preliminaries}
Throughout the paper, we assume that  $\Psi\in C^{1}(J,\mathbb{R})$ be an increasing function with $\Psi'(x)\neq 0$, $x\in J$. 
\begin{definition} [\cite{Kilbas}]
Let $\alpha>0$ and  $f$ an integrable function defined on $J$. Then $\Psi$-Riemann-Liouville fractional integrals of $f$ is given by  
\begin{equation*}\label{21}
\mathbb{I}_{0+}^{\alpha ;\, \Psi }f\left( t\right) :=\frac{1}{\Gamma \left( \alpha
\right) }\int_{0}^{t}\mathcal{A}_\Psi^\alpha(t,s) f\left( s \right) ds, ~\mbox{where}~ \mathcal{A}_\Psi^\alpha(t,s)=\Psi^{'}(s)(\Psi(t)-\Psi(s))^{\alpha-1}.
\end{equation*}
\end{definition}
Note that  for $\delta >0$, we have
\begin{small}
$\mathbb{I}_{0^+}^{\alpha;\, \Psi}(\Psi(t)-\Psi(0))^{\delta-1}=\frac{\Gamma(\delta)}{\Gamma(\alpha+\delta)}(\Psi(t)-\Psi(0))^{\alpha + \delta-1}.$
\end{small}
\begin{definition} [\cite{{Sousa1}}]
Let $f \in C^{1} \left( J,\,\mathbb{R}\right)$. The $\Psi$--Hilfer fractional derivative of a function $f$ of order $0<\alpha<1$ and type $0\leq \beta \leq 1$, is defined by
\begin{small}
$$
^\mathcal{H} \mathbb{D}^{\alpha, \, \beta; \, \Psi}_{0^+}f(t)= \mathbb{I}_{0^+}^{\beta ({1-\alpha});\, \Psi} \left(\frac{1}{{\Psi}^{'}(t)}\frac{d}{dt}\right)^{'}\mathbb{I}_{0^+}^{(1-\beta)(1-\alpha);\, \Psi} f(t).
$$
\end{small}
\end{definition}

\begin{theorem}[\cite{Sousa1}]\label{ab}
Let $f \in C^{1} \left( J,\,\mathbb{R}\right)$, $0<\alpha<1$ and $0\leq\beta \leq 1 $. Then
\begin{small}
\begin{enumerate}[topsep=0pt,itemsep=-1ex,partopsep=1ex,parsep=1ex]
\item[(i)] $\mathbb{I}_{0^+}^{\alpha;\, \Psi}\, {^\mathcal{H} \mathbb{D}^{\alpha, \, \beta; \, \Psi}_{0^+}}f(t)= f(t)- \mathcal{R}_\Psi^\rho (t,0) \,  \mathbb{I}_{0^+}^{(1-\beta)\,(1-\alpha);\,\Psi}f(0),$ where $\mathcal{R}_\Psi^\rho (t,0) = \frac{(\Psi(t)-\Psi(0))^{\rho-1}}{\Gamma(\rho)}$,
\item[(ii)] ${^\mathcal{H}\mathbb{D}^{\alpha, \, \beta; \, \Psi}_{0^+}}\,\mathbb{I}_{0^+}^{\alpha;\, \Psi}f(t)=f(t).$
\end{enumerate}
\end{small}
\end{theorem}
Consider the weighted space \cite{Sousa1} defined by
\begin{small}
 $$
C_{1-\rho;\, \Psi}(J,\,\mathbb{R})=\left\{u:(0,b]\to\mathbb{R} : ~(\Psi(t)-\Psi(0))^{1-\rho}u(t)\in C(J,\,\mathbb{R})\right\},~0< \rho\leq 1.
$$
\end{small}
Define the weighted space of piecewise continuous functions as
\begin{small}
\begin{align*}
{\mathcal{PC}}_{1-\rho; \,  \Psi}(J,\mathbb{R}) =  & \left\lbrace u:(0,b] \to\mathbb{R} :u\in C_{1-\rho;\, \Psi}((t_k,t_{k+1}],\mathbb{R}),~k=0,1,\cdots, p,\,\mathbb{I}_{0^+}^{1-\rho; \, \Psi} \, u(t_k^+), \right.\\
&  \left.~~ \mathbb{I}_{0^+}^{1-\rho; \, \Psi}\, u(t_k^-) ~\mbox{exists and}  ~\mathbb{I}_{0^+}^{1-\rho; \, \Psi}\, u(t_k^-)= \mathbb{I}_{0^+}^{1-\rho; \, \Psi} \, u(t_k) ~\mbox{for} ~ \, k = 1,\cdots, p \right\rbrace .
\end{align*}
\end{small}
Then $\mathcal{\mathcal{PC}}_{1-\rho; \, \Psi}(J,\mathbb{R})$ is a Banach space with the norm
\begin{small}
$
\|u\|_{\mathcal{\mathcal{PC}}_{1-\rho; \,  \Psi}(J,\mathbb{R})} = \sup_ {t\in J} \left|(\Psi(t)-\Psi(0))^{1-\rho}u(t)\right|.
$
\end{small}
Observe that for $\rho =1$, the space  
 $\mathcal{\mathcal{PC}}_{1-\rho; \,  \Psi}(J,\mathbb{R})$ reduces to $\mathcal{\mathcal{PC}}(J,\mathbb{R})$ which is dealt in \cite{fec,Benchohra3, Bai}.
 
  Next, we introduce the space 
 \begin{small}
 $$
 \mathbb{X}_{\mathcal{C},\,\rho,\,\Psi}=\{u:[-r,b]\to \mathbb{R}: u\in \mathcal{C}\cap \mathcal{\mathcal{PC}}_{1-\rho; \,  \Psi}(J,\mathbb{R}) \},
 $$
 \end{small}
with the norm 
\begin{small}
$
\|u\|_{\mathbb{X}_{\mathcal{C},\,\rho,\,\Psi}}=\max \left\{\|u\|_{\mathcal{C}},\, \|u\|_{\mathcal{\mathcal{PC}}_{1-\rho; \,  \Psi}(J,\mathbb{R})}\right\}.
$
\end{small}
One can verify that 
\begin{small}
( $\mathbb{X}_{\mathcal{C},\,\rho,\,\Psi},\,\|\cdot \|_{\mathbb{X}_{\mathcal{C},\,\rho,\,\Psi}}$) 
\end{small}
is a Banach space.

For $v\in \mathbb{X}_{\mathcal{C},\,\rho,\,\Psi} $ and $ \epsilon >0$ consider the following inequalities
\begin{small}
\begin{align}\label{HU1}
\begin{cases}
\left| ^\mathcal{H} \mathbb{D}^{\alpha,\beta;\,\Psi}_{0^+}  v(t)-f\left(t, v(t), v(h(t)), ^\mathcal{H} \mathbb{D}^{\alpha,\, \beta; \, \Psi}_{0^+}v(t) \right)\right| \leq \epsilon\, E_\alpha\left((\Psi(t)-\Psi(0))^\alpha \right),~ t \in J, \\
 |\Delta \mathbb{I}_{0^+}^{1-\rho; \, \Psi}v(t_k)- \mathcal{J}_k(v(t_k^-))| \leq \epsilon, ~k=1,\cdots,p, 
\end{cases}
\end{align}
\end{small}
where $E_\alpha$ is the Mittag--Leffler function \cite{Kilbas} defined by
\begin{small}
$$
E_\alpha(z)=\sum_{n=0}^{\infty}\frac{z^n}{\Gamma(n\alpha+1)},\, z\in \mathbb{C},\, Re(\alpha)>0.
$$
\end{small}
To examine the Ulam--Hyers--Mittag--Leffler (UHML) stability of the problem \eqref{p1}--\eqref{p2} we adopt the definitions given by Wang et. al. \cite{Zhang} and Liu et. al. \cite{Liu3}.
\begin{definition}
The equation \eqref{p1}--\eqref{p2} is said to be UHML stable with respect to $E_\alpha\left(\zeta_{f,\Psi} (\Psi(t)-\Psi(0))^\alpha \right)$, if for $\epsilon>0$ there exists a constant $C_{p,E_\alpha}>0$ and $\zeta_{f,\Psi}>0$such that, for every solution $v\in \mathbb{X}_{\mathcal{C},\,\rho,\,\Psi} $ of the inequality \eqref{HU1}, there is a unique solution $u\in \mathbb{X}_{\mathcal{C},\,\rho,\,\Psi}$ to the problem \eqref{p1}--\eqref{p4} satisfying
\begin{small} 
\begin{align*}
\begin{cases}
|u(t)-v(t)|=0,\, t\in [-r,0],  \\
 (\Psi(t)-\Psi(0))^{1-\rho}|u(t)-v(t)| \leq C_{p,E_\alpha}\,\epsilon \, E_\alpha\left(\zeta_{f,\Psi} (\Psi(t)-\Psi(0))^\alpha \right), ~\zeta_{f,\Psi}>0,~t\in J. 
\end{cases}
\end{align*}
\end{small}
 \end{definition}
  \begin{rem}\label{rm}
  We say that $v\in  \mathbb{X}_{\mathcal{C},\,\rho,\,\Psi} $ is the solution of the inequality \eqref{HU1} if there exist a function $\mathcal{E} \in  \mathbb{X}_{\mathcal{C},\,\rho,\,\Psi}$ and a sequence $\{\mathcal{E}_k\},\,k=1,\cdots,p$ (both depending on $v$) such that
  \begin{small}
  \begin{itemize}[topsep=0pt,itemsep=-1ex,partopsep=1ex,parsep=1ex]
 \item[\rm (1)]$|\mathcal{E}(t)|\leq \epsilon\,E_\alpha\left(\zeta_{f,\Psi} (\Psi(t)-\Psi(0))^\alpha \right),\, t\in J,\,\,|\mathcal{E}_k|\leq \epsilon,\,  ~k=1,\cdots,p, $
 \item[\rm (2)] $^\mathcal{H} \mathbb{D}^{\alpha,\, \beta; \, \Psi}_{0^+}v(t)=f\left(t, v(t), v(h(t)), ^\mathcal{H} \mathbb{D}^{\alpha,\, \beta; \, \Psi}_{0^+}v(t) \right)+\mathcal{E}(t),\,t\in J,$
 \item[\rm (3)]$\Delta \mathbb{I}_{0^+}^{1-\rho; \, \Psi}v(t_k)= \mathcal{J}_k(v(t_k^-))+\mathcal{E}_k,  ~k = 1,\cdots,p.$
  \end{itemize}
  \end{small}
  \end{rem}
 \begin{definition}[\cite{zhou}]
Let $(\mathcal{X},\rho)$ be a metric space. The operator $\mathcal{T}:\mathcal{X}\to \mathcal{X}$ is a Picard operator if there exists $x^* \in \mathcal{X}$ such that 
\begin{small}
\begin{itemize}[topsep=0pt,itemsep=-1ex,partopsep=1ex,parsep=1ex]
\item[1.] $\mathcal{F}_\mathcal{T}=\{x^*\}$, where $\mathcal{F}_\mathcal{T}=\{x\in \mathcal{X}: \mathcal{T}(x)=x\}$ ;
\item[2.] the sequence $\{\mathcal{T}^n (x_0)\}_{n\in \mathbb{N}}$ converges to $x^*$ for all $x_0 \in \mathcal{X}$.
\end{itemize}
\end{small}
 \end{definition}
 \begin{lemma}[\cite{rus}]\label{AGL} 
Let $(\mathcal{X},\rho, \leq)$ be an ordered metric space and let $\mathcal{T}:\mathcal{X}\to \mathcal{X}$ be an increasing Picard operator with $\mathcal{F}_\mathcal{T}=\{x_{\mathcal{T}}^*\}$. Then for any $x\in \mathcal{X}, x\leq \mathcal{T}(x)$ implies $x\leq x_{\mathcal{T}}^*$. 
 \end{lemma}

\begin{lemma}[\cite{Jose}]\label{JI}
Let $ \mathcal{U}\in {\mathcal{PC}}_{1-\rho;\, \psi}\left( J,\,\mathbb{R}\right) $ satisfy the following inequality
\begin{small}
$$
\mathcal{U}(t)\leq \mathcal{V}(t)+ \mathbf{g}(t)\,\int_{0}^{t}\mathcal{A}_\Psi^\alpha(t,s)\,\mathcal{U}(s)\, ds + \sum_{0<t_k<t}\beta_k \mathcal{U}(t_k^-),\, t>a,
$$
\end{small}
where $\mathbf{g}$ is a continuous function, $\mathcal{V}\in {\mathcal{PC}}_{1-\rho;\, \psi}\left( J,\,\mathbb{R}\right) $ is non-negative, $\beta_k > 0$ for $k=1,\cdots,p,$ then we have
\begin{small}
\begin{align*}
\mathcal{U}(t)\leq \mathcal{V}(t) &\left[\prod_{i=1}^{k}\left\lbrace 1+\beta_i E_\alpha \left(\mathbf{g}(t)\Gamma(\alpha)(\Psi(t_i)-\Psi(0))^\alpha \right)  \right\rbrace  \right] E_\alpha\left(\mathbf{g}(t)\Gamma(\alpha)(\Psi(t)-\Psi(0))^\alpha \right),\,\, t\in (t_k,t_{k+1}].
\end{align*} 
\end{small}
\end{lemma}
\section{Formula of solutions}
We need the following lemma to derive the equivalent fractional integral of the the impulsive problem \eqref{p1}-\eqref{p4}.
\begin{lemma} [\cite{JK}]\label{JK1}
Let $0<\alpha<1$ and $ h: \mathcal{J}\to \mathbb{R} $ be continuous. Then for any $b\in \mathcal{J} $ a function $u\in C_{1-\rho,\psi}\left( \mathcal{J},\,\mathbb{R}\right) $ defined  by 
\begin{small}
\begin{equation}\label{JK}
u(t)= \mathcal{R}_\Psi^\rho (t,0) \, \left. \left\{\mathbb{I}_{0^+}^{{1-\rho}; \, {\psi}}u(b)-\mathbb{I}_{0^+}^{1-\rho+\alpha; \, \psi}h(t)\right|_{t=b} \right\}+\mathbb{I}_{0^+}^{\alpha; \, \psi}h(t),
\end{equation}
\end{small}
is the solution of the $\psi$--Hilfer fractional differential equation $^\mathcal{H} \mathbb{D}^{\alpha, \, \beta; \, \psi}_{0^+}u(t)=h(t),~t \in \mathcal{J}.$
\end{lemma}

\begin{theorem}\label{fie}
A function $u \in \mathbb{X}_{\mathcal{C},\,\rho,\,\Psi}$ is a solution of implicit impulsive \eqref{p1}-\eqref{p4} if and only if u is a solution of the following fractional integral equation
\begin{small}
\begin{equation}\label{ie1}
u(t) =
\begin{cases}
\phi(t),~ t\in [-r,0], \\
 \mathcal{R}_\Psi^\rho (t,0) \, \, \left(u_0 +\sum_{0<t_k <t}\mathcal{J}_k(u(t_k^-))\right)+\mathbb{I}_{0^+}^{\alpha; \, \psi}g_u(t), ~t\in J,  
\end{cases}
\end{equation}
\end{small}
where
\begin{small}
$$
g_u(t)=f\left( t,u(t), u(h(t)),g_u(t)\right) .
$$
\end{small}
\end{theorem}
\begin{proof}
Assume that $u \in  \mathbb{X}_{\mathcal{C},\,\rho,\,\Psi}$ satisfies the implicit impulsive $\psi$--HFDE \eqref{p1}-\eqref{p4}.
If $t\in [0,t_1]$ then 
\begin{small}
\begin{equation} \label{ie2}
\begin{cases}
^\mathcal{H} \mathbb{D}^{\alpha, \, \beta; \, \psi}_{0^+}u(t)=f\left( t,u(t), u(h(t)),^\mathcal{H} \mathbb{D}^{\alpha, \, \beta; \, \psi}_{0^+}u(t)\right), \\
\mathbb{I}_{0^+}^{1-\rho; \, \psi}u(0)=u_0.  
\end{cases} 
\end{equation}
\end{small}
Let $^\mathcal{H} \mathbb{D}^{\alpha, \, \beta; \, \psi}_{0^+}u(t)=g_u(t).$  Then we have
$g_u(t)=f\left( t,u(t), u(h(t)),g_u(t)\right) $ and \eqref{ie2} becomes
\begin{small}
\begin{equation} \label{ie3}
\begin{cases}
^\mathcal{H} \mathbb{D}^{\alpha, \, \beta; \, \psi}_{0^+}u(t)=g_u(t),\\
\mathbb{I}_{0^+}^{1-\rho; \, \psi}u(0)=u_0.  
\end{cases} 
\end{equation}
\end{small}
Then the problem \eqref{ie3} is equivalent to the following fractional integral \cite{Sousa2}
\begin{small}
\begin{equation}\label{ie4}
u(t)= \mathcal{R}_\Psi^\rho (t,0) \,  \, u_0 +\mathbb{I}_{0^+}^{\alpha; \, \psi}g_u(t) ,\, \text{ $t \in [0,t_1] $.}
 \end{equation}
\end{small} 
Now, if $t\in (t_1,t_2]$ then in the view of \eqref{ie2} we have
\begin{small}
$$
^\mathcal{H} \mathbb{D}^{\alpha, \, \beta; \, \psi}_{0^+}u(t)=g_u(t), \, t\in (t_1,t_2] 
~~\mbox{with}~~\mathbb{I}_{0^+}^{1-\rho; \, \psi}u(t_1^+)- \mathbb{I}_{0^+}^{1-\rho; \, \psi}u(t_1^-)=\mathcal{J}_1(u(t_1^-)).
$$
\end{small}
By Lemma \ref{JK1}, we have
\begin{small}
\begin{align}\label{ie5}
u(t)\nonumber &= \mathcal{R}_\Psi^\rho (t,0) \, \left. \left\{\mathbb{I}_{0^+}^{{1-\rho}; \, {\psi}}u(t_1^+)-\mathbb{I}_{0^+}^{1-\rho+\alpha; \, \psi}g_u(t)\right|_{t=t_1} \right\}+\mathbb{I}_{0^+}^{\alpha; \, \psi}g_u(t) \nonumber \\ 
&=  \mathcal{R}_\Psi^\rho (t,0) \,  \left. \left\{\mathbb{I}_{0^+}^{{1-\rho}; \, {\psi}}u(t_1^-)+\mathcal{J}_1(u(t_1^-))-\mathbb{I}_{0^+}^{1-\rho+\alpha; \, \psi}g_u(t)\right|_{t=t_1} \right\}\nonumber \\
& \qquad +\mathbb{I}_{0^+}^{\alpha; \, \psi}g_u(t), ~ t\in (t_1,t_2].
   \end{align}
\end{small}
Now, from \eqref{ie4}, we have 
$ \mathbb{I}_{0^+}^{{1-\rho}; \, {\psi}}u(t)= u_0+\mathbb{I}_{0^+}^{{1-\rho+\alpha}; \, {\psi}}g_u(t).$
This gives
\begin{small}
\begin{equation}\label{ie6}
\left. \mathbb{I}_{0^+}^{{1-\rho}; \, {\psi}}u(t_1^-)-\mathbb{I}_{0^+}^{1-\rho+\alpha; \, \psi}g_u(t)\right|_{t=t_1}=u_0.
\end{equation}
\end{small}
Using \eqref{ie6} in \eqref{ie5}, we obtain
\begin{small}
 \begin{equation}\label{ie7}
u(t)=\mathcal{R}_\Psi^\rho (t,0) \,  \left(u_0+{\mathcal{J}}_1(u(t_1^-))\right) +\mathbb{I}_{0^+}^{\alpha; \, \psi}g_u(t), ~ t\in (t_1,t_2].
\end{equation} 
\end{small}
Continuing in this manner, we obtain
\begin{small}
\begin{equation}\label{ie11}
 u(t)=\mathcal{R}_\Psi^\rho (t,0) \, \, \left(u_0 +\sum_{i=1}^{k} \mathcal{J}_{i}(u(t_i^-))\right)+\mathbb{I}_{0^+}^{\alpha; \, \psi}g_u(t),~t\in (t_k, t_{k+1}],~k=1,\cdots,p.
\end{equation}
\end{small}
From above we obtain \eqref{ie1}.

Conversely, let $u \in \mathbb{X}_{\mathcal{C},\,\rho,\,\Psi}$ satisfies the fractional integral equation \eqref{ie1}. Then, for $t\in J$, we have
\begin{small}
$$
u(t)=\mathcal{R}_\Psi^\rho (t,0) \, \, \left(u_0 +\sum_{0<t_k <t}\mathcal{J}_k(u(t_k^-))\right)+\mathbb{I}_{0^+}^{\alpha; \, \psi}g_u(t), ~t\in J. $$
\end{small}
Applying the $\psi$-Hilfer fractional derivative operator $^\mathcal{H} D^{\alpha, \, \beta; \, \psi}_{0^+}$ on both sides of the above equaiton and using the result (\cite{Sousa2}, Page 10),  
$
^\mathcal{H} \mathbb{D}^{\alpha, \, \beta; \, \psi}_{0^+}(\psi(t)-\psi(0))^{\rho-1}=0, ~ 0< \rho <1,
$
Theorem \ref{ab}, we obtain
\begin{small}
\begin{equation}\label{ie8}
^\mathcal{H} \mathbb{D}^{\alpha,\, \beta; \, \Psi}_{0^+}u(t)=g_u(t)=f\left( t,u(t), u(h(t)),^\mathcal{H} \mathbb{D}^{\alpha,\, \beta; \, \Psi}_{0^+}u(t)\right) ,~t \in J.
\end{equation}
\end{small}
which is \eqref{p1}. Further, from \eqref{ie4}, we have
\begin{small}
\begin{align*}
\mathbb{I}_{0^+}^{{1-\rho}; \, {\psi}}u(t)&=u_0 \,\,\mathbb{I}_{0^+}^{{1-\rho}; \, {\psi}}\mathcal{R}_\Psi^\rho (t,0) \, + \mathbb{I}_{0^+}^{{1-\rho}; \, {\psi}}\, \mathbb{I}_{0^+}^{\alpha; \, \psi}g_u(t)= u_0 + \mathbb{I}_{0^+}^{{1-\rho+\alpha}; \, {\psi}} g_u(t),
\end{align*}
\end{small}
which gives
\begin{equation*}\label{ie10}
\mathbb{I}_{0^+}^{{1-\rho}; \, {\psi}}u(0)=u_0, 
\end{equation*}
Now from equation \eqref{ie11},  for $ t\in(t_k,t_{k+1}]$,  we have
\begin{small}
\begin{align}\label{ie12} 
\mathbb{I}_{0^+}^{{1-\rho}; \, {\psi}}u(t)&=\left\{u_0+\sum_{i=1}^{k}{\mathcal{J}}_i (u(t_i^-))\right \}\mathbb{I}_{0^+}^{{1-\rho}; \, {\psi}}\mathcal{R}_\Psi^\rho (t,0) \, + \mathbb{I}_{0^+}^{{1-\rho}; \, {\psi}} \mathbb{I}_{0^+}^{\alpha; \, \psi}g_u(t) \nonumber\\
&= u_0+\sum_{i=1}^{k}{\mathcal{J}}_i (u(t_i^-)) + \mathbb{I}_{0^+}^{{1-\rho+\alpha}; \, {\psi}} g_u(t). 
\end{align}
\end{small}  
Again for  $ t\in(t_{k-1},t_k]$,  we have
\begin{small}
\begin{align} \label{ie13}
\mathbb{I}_{0^+}^{{1-\rho}; \, {\psi}}u(t)&= \left\{u_0+\sum_{i=1}^{k-1}{\mathcal{J}}_i (u(t_i^-))\right \}\mathbb{I}_{0^+}^{{1-\rho}; \, {\psi}}\mathcal{R}_\Psi^\rho (t,0) \, + \mathbb{I}_{0^+}^{{1-\rho}; \, {\psi}} \mathbb{I}_{0^+}^{\alpha; \, \psi}g_u(t) \nonumber\\
&= u_0+\sum_{i=1}^{k-1}{\mathcal{J}}_i (u(t_i^-)) + \mathbb{I}_{0^+}^{{1-\rho+\alpha}; \, {\psi}} g_u(t),
\end{align}
\end{small}
Therefore, from \eqref{ie12} to \eqref{ie13}, we obtain
\begin{small}
\begin{equation}\label{ie14}
\mathbb{I}_{0^+}^{1-\rho; \, \psi}u(t_k^+)- \mathbb{I}_{0^+}^{1-\rho; \, \psi}u(t_k^-)= \sum_{i=1}^{k}{\mathcal{J}}_i (u(t_i^-))-\sum_{i=1}^{k-1}{\mathcal{J}}_i (u(t_i^-)) = {\mathcal{J}}_k (u(t_k^-)).
\end{equation}
\end{small}
 This completes the proof.
\end{proof}
\section{Existence, Uniqueness and UHML Stabilty}

This section deals with the
In this section, we derive the existence and uniqueness of solution to the problem  \eqref{p1}--\eqref{p4}. Further, the UHML stability of the equation \eqref{p1}--\eqref{p2} is investigated.
\begin{theorem}\label{exi}
Assume that:
\begin{itemize}[topsep=0pt,itemsep=-1ex,partopsep=1ex,parsep=1ex]
\item[($H_1$)] The function $f:J\times \mathbb{R}^{3} \to \mathbb{R}$ is continuous and there exists constant $K>0$ and $0< \mathcal{L}_f <1 $ satisfy the following condition:
\begin{small}
$$
\left| f(t,u_1, u_2, u_3)-f(t,v_1, v_2, v_3)\right|\leq K (\Psi(t)-\Psi(0))^{1-\rho}\,\sum_{i=1}^{2}|u_i-v_i| + \mathcal{L}_f |u_3 - v_3|,
$$
 $t\in J$ and $u_i, v_i \in \mathbb{R} ~\text{for}\,~i=1,2,3.$
 \end{small}
\item[($H_2$)] The functions $\mathcal{J}_k:\mathbb{R}\to \mathbb{R}$, $(k=1,\cdots,p)$ satisfy the condition
\begin{small}
$$
\left| \mathcal{J}_k(u(t_k^-))-\mathcal{J}_k(v(t_k^-))\right|\leq \mathcal{L}_{{\mathcal{J}}_{k}} \,(\Psi(t_k^-)-\Psi(0))^{1-\rho}\,|u(t_k^-)-v(t_k^-)|,
$$
~where  $u\in {\mathcal{PC}}_{1-\rho;\, \Psi}\left( J,\,\mathbb{R}\right)\,and \, \mathcal{L}_{{\mathcal{J}}_{k}} >0.$
\item[($H_3$)]  
$
\mathcal{L}:\frac{\sum_{k=1}^{p}\mathcal{L}_{{\mathcal{J}}_{k}}}{\Gamma(\rho)}+\frac{2K(\Psi(b)-\Psi(0))^{\alpha-\rho+1}}{(1-\mathcal{L}_f)\Gamma(\alpha+1)}<1.
$
\end{small}
\end{itemize}
Then,
\begin{itemize}[topsep=0pt,itemsep=-1ex,partopsep=1ex,parsep=1ex]
\item[1.] the problem \eqref{p1}--\eqref{p4} has unique solution in the space $\mathbb{X}_{\mathcal{C},\,\rho,\,\Psi}$ ;
\item[2.] the equation \eqref{p1}--\eqref{p2} is UHML stable.
\end{itemize}
 \end{theorem}
\begin{proof}
(1) By Theorem \ref{fie}, the equivalent fractional integral equation of the problem  \eqref{p1}--\eqref{p4} is given by the equation \eqref{ie1}.
Define the operator
\begin{small}
 $\mathbb{T}: ( \mathbb{X}_{\mathcal{C},\,\rho,\,\Psi},\, \|\cdot\|_{\mathbb{X}_{\mathcal{C},\,\rho,\,\Psi}}) \to ( \mathbb{X}_{\mathcal{C},\,\rho,\,\Psi},\, \|\cdot\|_{\mathbb{X}_{\mathcal{C},\,\rho,\,\Psi}})$ 
  by
\begin{equation}\label{e16}
\mathbb{T}(u(t))=
\begin{cases}
\phi(t), ~~~ t\in [-r,0],\\
 \mathcal{R}_\Psi^\rho (t,0) \, \, \left(u_0 +\sum_{a<t_k<t} \mathcal{J}_k(u(t_k^-))\right)+\mathbb{I}_{0^+}^{\alpha; \, \Psi}g_u(t), ~\text{ $t \in J$},  
\end{cases}
\end{equation}
\end{small}
where 
\begin{small}
\begin{equation}\label{e616}
g_u(t)=f\left( t,u(t), u(h(t)),g_u(t)\right).
\end{equation}
\end{small}
Then the solution of \eqref{p1}--\eqref{p4} will be the fixed point of $\mathbb{T}$.
In order to prove $\mathbb{T}$ is Picard operator, we prove that $\mathbb{T}$ is contraction mapping. Let any $u, \tilde{u}\in \mathbb{X}_{\mathcal{C},\,\rho,\,\Psi}$. Then for any $t \in [-r,0]$, 
\begin{small}
\begin{equation}\label{e20}
\left| \mathbb{T}(u(t))-\mathbb{T}(\tilde{u}(t))\right| =0 \implies \|\mathbb{T}(u(t))-\mathbb{T}(\tilde{u}(t))\|_{\mathcal{C}}=0.
\end{equation}
\end{small}
Further, for any $t\in J$, by definition of $\mathbb{T}$ we have 
\begin{small}
\begin{align} \label{e717}
|\mathbb{T}(u(t))-\mathbb{T}(\tilde{u}(t))| 
&\leq  \mathcal{R}_\Psi^\rho (t,0) \, \sum_{a<t_k<t}\left| \mathcal{J}_k(u(t_k^-))-\mathcal{J}_k(\tilde{u}(t_k^-))\right| \nonumber \\
&\qquad + \frac{1}{\Gamma(\alpha)} \int_{0}^{t}\mathcal{A}_\Psi^\alpha(t,s)\,|g_u (s)-g_{\tilde{u}}(s)| \, ds.
\end{align}
\end{small}
Using ($H_1$) and \eqref{e616}, for any $t\in J $ we have,
\begin{small}
$$
|g_u (t)-g_{\tilde{u}}(t)|
 \leq K  (\Psi(t)-\Psi(0))^{1-\rho}\left\lbrace |u(t)-\tilde{u}(t)|+|u(h(t))-\tilde{u}(h(t))| \right\rbrace + \mathcal{L}_f |g_u (t)-g_{\tilde{u}}(t)|.
$$
\end{small}
This implies that
\begin{small}
\begin{equation}\label{e818}
|g_u (t)-g_{\tilde{u}}(t)| \leq \frac{K (\Psi(t)-\Psi(0))^{1-\rho} }{1-\mathcal{L}_f} \left\lbrace |u(t)-\tilde{u}(t)|+|u(h(t))-\tilde{u}(h(t))| \right\rbrace. 
\end{equation}
\end{small}
Making use of hypothesis $(H_2)$ and the inequality \eqref{e818},\,\eqref{e717} takes the form,
\begin{small}
\begin{align*}
(\Psi(t)-\Psi(0))^{1-\rho}|\mathbb{T}(u(t))-\mathbb{T}(\tilde{u}(t))|
& \leq \frac{1}{\Gamma(\rho)} \sum_{0<t_k<t} \mathcal{L}_{{\mathcal{J}}_{k}} (\Psi(t_k^-)-\Psi(0))^{1-\rho}\,|u(t_k^-)-v(t_k^-)|\\
&+ \frac{K\,(\Psi(t)-\Psi(0))^{1-\rho}}{(1-\mathcal{L}_f)\Gamma(\alpha)} \int_{a}^{t}\mathcal{A}_\Psi^\alpha(t,s)\,(\Psi(s)-\Psi(0))^{1-\rho}\\
&\qquad \qquad \times \left\lbrace |u(s)-\tilde{u}(s)|+|u(h(s))-\tilde{u}(h(s))| \right\rbrace \, ds\\
&\leq \frac{1}{\Gamma(\rho)} \sum_{k=1}^{p} \mathcal{L}_{{\mathcal{J}}_{k}} \|u-\tilde{u}\|_{{\mathcal{PC}}_{1-\rho;\, \psi}(J,\mathbb{R})}\\ 
&+\frac{2\,K \,(\Psi(b)-\Psi(0))^{1-\rho}}{(1-\mathcal{L}_f)\,\Gamma(\alpha)} \,\|u-\tilde{u}\|_{{\mathcal{PC}}_{1-\rho;\, \psi}(J,\mathbb{R})}\,\int_{0}^{t}\mathcal{A}_\Psi^\alpha(t,s)\,ds  \\
&\leq \left( \frac{\sum_{k=1}^{p} \mathcal{L}_{{\mathcal{J}}_{k}}}{\Gamma(\rho)} + \frac{2\,K \,(\Psi(b)-\Psi(0))^{1-\rho+\alpha}}{(1-\mathcal{L}_f)\Gamma(\alpha+1)} \right)\|u-\tilde{u}\|_{{\mathcal{PC}}_{1-\rho;\, \psi}(J,\mathbb{R})}.   
\end{align*}
\end{small}
Therefore,
\begin{small}
\begin{align}\label{e111}
\|\mathbb{T}(u(t))-\mathbb{T}(\tilde{u}(t))\|_{{\mathcal{PC}}_{1-\rho;\, \psi}(J,\mathbb{R})}
& = \sup_{t\in J} \left|(\Psi(t)-\Psi(0))^{1-\rho}(\mathbb{T}(u(t))-\mathbb{T}(\tilde{u}(t)))\right| \leq \mathcal{L}\, \|u-\tilde{u}\|_{{\mathcal{PC}}_{1-\rho;\, \psi}(J,\mathbb{R})}.
\end{align}
From \eqref{e20} and \eqref{e111} we have,
\begin{align*}
\|\mathbb{T}(u(t))-\mathbb{T}(\tilde{u}(t))\|_{\mathbb{X}_{\mathcal{C},\,\rho,\,\Psi}}
&= \max \left\{\|\mathbb{T}(u(t))-\mathbb{T}(\tilde{u}(t))\|_{\mathcal{C}},\,\|\mathbb{T}(u(t))-\mathbb{T}(\tilde{u}(t))\|_{{\mathcal{PC}}_{1-\rho;\, \psi}(J,\mathbb{R})} \right\}\\
&\leq \mathcal{L}\, \max \left\{0,\, \|u-\tilde{u}\|_{{\mathcal{PC}}_{1-\rho;\, \psi}(J,\mathbb{R})}\right\}\\
&\leq \mathcal{L}\, \|u-\tilde{u}\|_{\mathbb{X}_{\mathcal{C},\,\rho,\,\Psi}}.
\end{align*}
\end{small}
Since $\mathcal{L}<1$, $\mathbb{T} $ is a contraction on $\mathbb{X}_{\mathcal{C},\,\rho,\,\Psi}$.Hence by Banach contraction principle $\mathbb{T}$ has a unique fixed point in $\mathbb{X}_{\mathcal{C},\,\rho,\,\Psi}$.
which is the unique solution of \eqref{p1}--\eqref{p4}.
 
(2) In this part we prove that the problem \eqref{p1}--\eqref{p2} is UHML stable.
 let $v\in \mathbb{X}_{\mathcal{C},\,\rho,\,\Psi}$ be a solution of the inequality \eqref{HU1}.Then by the Theorem \ref{fie}, and Remark \ref{rm}, we have
 \begin{small}
 \begin{equation}\label{e222}
 v(t) =
   \mathcal{R}_\Psi^\rho (t,0) \, \, \left(\mathbb{I}_{0^+}^{1-\rho; \, \Psi}v(0) +\sum_{0<t_k<t}\left(  \mathcal{J}_k(v(t_k^-))+\mathcal{E}_k\right) \right)+\mathbb{I}_{0^+}^{\alpha; \, \Psi}g_v(t) + \mathbb{I}_{0^+}^{\alpha; \, \Psi} \mathcal{E}, ~\text{ $t \in J$},  
   \end{equation}
   \end{small}
 where
 \begin{small}
 $
 g_v(t)= f(t,v(t), v(h(t)),g_v(t)).
 $
 \end{small}
 
Let $u\in \mathbb{X}_{\mathcal{C},\,\rho,\,\Psi}$ be the unique solution of the problem
\begin{small}
\begin{align}\label{e333}
\begin{cases}
 ^\mathcal{H} \mathbb{D}^{\alpha,\, \beta; \, \Psi}_{0^+}u(t)=f\left(t, u(t), u(h(t)), ^\mathcal{H} \mathbb{D}^{\alpha,\, \beta; \, \Psi}_{0^+}u(t) \right) ,~t \in J-\{t_1, t_2,\cdots ,t_p\},\\
\Delta \mathbb{I}_{0^+}^{1-\rho; \, \Psi}u(t_k)= \mathcal{J}_k(u(t_k^-)),  ~k = 1,\cdots,p, \\
 \mathbb{I}_{0^+}^{1-\rho; \, \Psi}u(0)=\mathbb{I}_{0^+}^{1-\rho; \, \Psi}v(0), \\
 u(t)=v(t),\, t\in [-r,0]. 
\end{cases}
\end{align}
\end{small} 
Then from the equation \eqref{e222} and in the view of Remark \ref{rm}, for any $t\in J$, we have
\begin{small}
\begin{align}\label{e21}
& \left|v(t) - \mathcal{R}_\Psi^\rho (t,0) \, \, \left(\mathbb{I}_{0^+}^{1-\rho; \, \Psi}v(0) +\sum_{0<t_k<t}  \mathcal{J}_k(v(t_k^-))\right)-\mathbb{I}_{0^+}^{\alpha; \, \Psi}g_v(t)\right| \nonumber \\
& \leq \mathcal{R}_\Psi^\rho (t,0) \,  \sum_{k=1}^{p}\left|\mathcal{E}_k \right| + \frac{\epsilon}{\Gamma(\alpha)} \int_{0}^{t}\mathcal{A}_\Psi^\alpha(t,s)\,|\mathcal{E}(s)| \, ds\nonumber \\
& \leq m \epsilon \,  \mathcal{R}_\Psi^\rho (t,0) \, +\frac{\epsilon}{\Gamma(\alpha)} \int_{0}^{t}\mathcal{A}_\Psi^\alpha(t,s)\,E_{\alpha}((\Psi(s)-\Psi(0))^\alpha)\, ds \nonumber\\
&\leq m \epsilon \,  \mathcal{R}_\Psi^\rho (t,0) \,  + \frac{\epsilon}{\Gamma(\alpha)} \sum_{n=0}^{\infty} \frac{1}{\Gamma(n\alpha+1)}\int_{0}^{t}\mathcal{A}_\Psi^\alpha(t,s)\,(\Psi(s)-\Psi(0)^{n\alpha})\, ds \nonumber\\
& = m \epsilon \,  \mathcal{R}_\Psi^\rho (t,0) \,  +\frac{\epsilon}{\Gamma(\alpha)} \sum_{n=0}^{\infty} \frac{1}{\Gamma(n\alpha+1)}(\Psi(t)-\Psi(0))^{\alpha} \nonumber\\
&\qquad \times \int_{0}^{t}\Psi^{'}(s)\left(1-\frac{\Psi(s)-\Psi(0)}{\Psi(t)-\Psi(0)} \right) ^{\alpha-1}\,(\Psi(s)-\Psi(0)^{n\alpha})\, ds \nonumber\\
&= m \epsilon \,  \mathcal{R}_\Psi^\rho (t,0) \, +\frac{\epsilon}{\Gamma(\alpha)} \sum_{n=0}^{\infty} \frac{1}{\Gamma(n\alpha+1)}(\Psi(t)-\Psi(0))^{n\alpha+\alpha}\int_{0}^{1}(1-\theta)^{\alpha-1} \theta ^ {n\alpha}\, ds \nonumber\\
&\qquad\qquad  \text{letting}\, \theta  = \frac{\Psi(s)-\Psi(0)}{\Psi(t)-\Psi(0)} \,~ \text{we have,}\,~ \Psi^{'}(s)\, ds = (\Psi(t)-\Psi(0))\, d\theta \nonumber\\
& =m \epsilon \,  \mathcal{R}_\Psi^\rho (t,0) \, +\frac{\epsilon}{\Gamma(\alpha)} \sum_{n=0}^{\infty} \frac{(\Psi(t)-\Psi(0))^{n\alpha+\alpha}}{\Gamma(n\alpha+1)}\,\frac{\Gamma(\alpha)\,\Gamma(n \alpha+1)}{\Gamma(\alpha+n\alpha+1)} \nonumber\\
&= m \epsilon \,  \mathcal{R}_\Psi^\rho (t,0) \, +\frac{\epsilon}{\Gamma(\alpha)} \sum_{n=0}^{\infty} \frac{(\Psi(t)-\Psi(0))^{(n+1)\alpha}}{\Gamma((n+1)\alpha)} \nonumber\\
&\leq m \epsilon \,  \mathcal{R}_\Psi^\rho (t,0) \,  + \epsilon\, E_\alpha ((\Psi(t)-\Psi(0))^\alpha).
\end{align}
\end{small}
Now for $t\in [-r,0],\, |v(t)-u(t)|=0.$ Further, utilizing ($H_2$), \eqref{e818} and \eqref{e21}, for any $t\in J,$ we have
\begin{small}
\begin{align*}
|v(t)-u(t)|
& \leq \left|v(t) - \mathcal{R}_\Psi^\rho (t,0) \, \, \left(\mathbb{I}_{0^+}^{1-\rho; \, \Psi}v(0) +\sum_{0<t_k<t}  \mathcal{J}_k(v(t_k^-))\right)-\mathbb{I}_{0^+}^{\alpha; \, \Psi}g_v(t)\right|\\
&\qquad + \mathcal{R}_\Psi^\rho (t,0) \,  \sum_{0<t_k<t}  \left| \mathcal{J}_k(v(t_k^-))-  \mathcal{J}_k(u(t_k^-))\right| + \left|\mathbb{I}_{0^+}^{\alpha; \, \Psi}\left( g_v(t)- g_u(t) \right) \right| \\
&\leq \left( m \epsilon \,  \mathcal{R}_\Psi^\rho (t,0) \,  + \epsilon\, E_\alpha ((\Psi(t)-\Psi(0))^\alpha)\right) \\
&\qquad +  \mathcal{R}_\Psi^\rho (t,0) \,  \sum_{0<t_k<t}   \mathcal{L}_{{\mathcal{J}}_{k}} (\Psi(t_k^-)-\Psi(0))^{1-\rho}\,|v(t_k^-)-u(t_k^-)|\\
&\qquad +  \frac{K}{(1-\mathcal{L}_f)\Gamma(\alpha)} \int_{0}^{t}\mathcal{A}_\Psi^\alpha(t,s)\,(\Psi(s)-\Psi(0))^{1-\rho}\\
&\qquad \qquad \times \left\lbrace |v(s)-u(s)|+|v(h(s))-u(h(s))| \right\rbrace \, ds.
\end{align*}
This gives
\begin{align}\label{e22}
(\Psi(t)-\Psi(0))^{1-\rho}|v(t)-u(t)|
&\leq \left( \frac{m \epsilon }{\Gamma(\rho)} + \epsilon\,(\Psi(b)-\Psi(0))^{1-\rho}\, E_\alpha ((\Psi(b)-\Psi(0))^\alpha)\right) \nonumber\\
&\qquad+ \frac{1}{\Gamma(\rho)} \sum_{0<t_k<t}   \mathcal{L}_{{\mathcal{J}}_{k}} (\Psi(t_k^-)-\Psi(0))^{1-\rho}\,|v(t_k^-)-u(t_k^-)| \nonumber\\
&\qquad \qquad +  \frac{K\,(\Psi(b)-\Psi(0))^{1-\rho}}{(1-\mathcal{L}_f)\Gamma(\alpha)} \int_{0}^{t}\mathcal{A}_\Psi^\alpha(t,s)\,(\Psi(s)-\Psi(0))^{1-\rho} \nonumber\\
&\qquad \qquad \qquad\times \left\lbrace |v(s)-u(s)|+|v(h(s))-u(h(s))| \right\rbrace \, ds.
\end{align}
\end{small}
Next, we consider the Banach space $\mathcal{B}=C([-r,b],\, \mathbb{R}_+)$ of all continuous functions $z: [-r,b]\to \mathbb{R}_+$ endowed with the supremum norm $\|z\|_{\mathcal{B}}=\sup_{t\in [-r,b]}|z(t)|$. Define the operator $\mathcal{Q}:\mathcal{B}\to \mathcal{B}$ by 
\begin{small}
\begin{equation}\label{e23}
(\mathcal{Q} z)(t) =
\begin{cases}
~0,  ~ t\in [-r,0],\\
  \left( \frac{m \epsilon }{\Gamma(\rho)} + \epsilon\,(\Psi(b)-\Psi(0))^{1-\rho}\, E_\alpha ((\Psi(b)-\Psi(0))^\alpha)\right)+ \frac{1}{\Gamma(\rho)} \sum_{0<t_k<t}   \mathcal{L}_{{\mathcal{J}}_{k}} \,w(t_k^-)\\
  \qquad \qquad +  \frac{K\,(\Psi(b)-\Psi(0))^{1-\rho}}{(1-\mathcal{L}_f)\Gamma(\alpha)} \int_{0}^{t}\mathcal{A}_\Psi^\alpha(t,s)\,\left(  z(s)+z(h(s)) \right)  \, ds , ~ t\in J.  
\end{cases}
\end{equation}
\end{small}
We prove that $\mathcal{Q}$ is a Picard operator. Let $z,\tilde{z} \in \mathcal{B} .$
Then
\begin{small} 
\begin{equation}\label{e444}
|\mathcal{Q}z(t)-\mathcal{Q}\tilde{z}(t)|=0,\,t \in [-r,0].
\end{equation}
\end{small}
Now for any $t\in J$,
\begin{small}
\begin{align}\label{e555}
\left|\mathcal{Q}z(t)-\mathcal{Q}\tilde{z}(t) \right| 
& \leq \frac{1}{\Gamma(\rho)} \sum_{0<t_k<t}   \mathcal{L}_{{\mathcal{J}}_{k}} |z(t_k)-\tilde{z}(t_k)|\nonumber\\
&\qquad+ \frac{K\,(\Psi(b)-\Psi(0))^{1-\rho}}{(1-\mathcal{L}_f)\Gamma(\alpha)} \left( \int_{0}^{t}\mathcal{A}_\Psi^\alpha(t,s)|z(s)-\tilde{z}(s)|\, ds \right.\nonumber\\ 
&\qquad \qquad \left. +\int_{0}^{t}\mathcal{A}_\Psi^\alpha(t,s)|z(h(s))-\tilde{z}(h(s))|\, ds \right)\nonumber\\
& \leq \left( \frac{\sum_{k=1}^{p}\mathcal{L}_{\mathcal{J}_k}}{\Gamma(\rho)} + \frac{2\, K\,(\Psi(b)-\Psi(0))^{1-\rho+\alpha}}{(1-\mathcal{L}_f) \,\Gamma(\alpha+1)} \right) \, \|z-\tilde{z}\|_{\mathcal{B}}.
\end{align}
\end{small}
From \eqref{e444} and \eqref{e555} it follows that
\begin{small} 
$$
\|\mathcal{Q}z-\mathcal{Q}\tilde{z}\|_{\mathcal{B}}\leq\left( \frac{\sum_{k=1}^{p}\mathcal{L}_{\mathcal{J}_k}}{\Gamma(\rho)} + \frac{2\, K\,(\Psi(b)-\Psi(0))^{1-\rho+\alpha}}{(1-\mathcal{L}_f) \,\Gamma(\alpha+1)} \right) \|z-\tilde{z}\|_{\mathcal{B}}.
$$
By ($H_3$), $\left( \frac{\sum_{k=1}^{p}\mathcal{L}_{\mathcal{J}_k}}{\Gamma(\rho)} + \frac{2\, K\,(\Psi(b)-\Psi(0))^{1-\rho+\alpha}}{(1-\mathcal{L}_f) \,\Gamma(\alpha+1)} \right)<1$.
\end{small}
Therefore, $\mathcal{Q}$ is contraction.
By Banach contraction principle $F_{\mathcal{Q}}=\{z^*\}.$ It follows that 
\begin{small}
\begin{align}\label{e24}
z^*(t)
& =  \left( \frac{m \epsilon }{\Gamma(\rho)} + \epsilon\,(\Psi(b)-\Psi(0))^{1-\rho}\, E_\alpha ((\Psi(b)-\Psi(0))^\alpha)\right)  + \frac{1}{\Gamma(\rho)} \sum_{0<t_k<t}   \mathcal{L}_{{\mathcal{J}}_{k}} z^* (t_k^-) \nonumber\\
 &\qquad +  \frac{K\,(\Psi(b)-\Psi(0))^{1-\rho}}{(1-\mathcal{L}_f)\Gamma(\alpha)} \int_{0}^{t}\mathcal{A}_\Psi^\alpha(t,s)\,\left(  z^* (s)+z^* (h(s)) \right)  \, ds.
\end{align}
\end{small}
Next we show that $z^*$ is increasing. Let any $t_1, t_2 \in [-r,b]$ with $t_1 < t_2.$ If $t_1, t_2 \in [-r,0]$ then $z^*(t_2)-z^*(t_1)=0.$ Let $0< t_1 < t_2 \leq b$
Define $\mathcal{M}= \min_{s\in [0,b]} \left(z^* (s)+z^* (h(s))\right)$. Then
\begin{small}
\begin{align*}
 z^*(t_2)-z^*(t_1)
&= \frac{1}{\Gamma(\rho)} \sum_{0<t_k<t_2}   \mathcal{L}_{{\mathcal{J}}_{k}} w^* (t_k^-) - \frac{1}{\Gamma(\rho)} \sum_{0<t_k<t_1}   \mathcal{L}_{{\mathcal{J}}_{k}} z^* (t_k^-)\\
& \qquad + \frac{K\,(\Psi(b)-\Psi(0))^{1-\rho}}{(1-\mathcal{L}_f)\Gamma(\alpha)}\left(  \int_{0}^{t_2}\Psi^{'}(s)(\Psi(t_2)-\Psi(s))^{\alpha-1}\,\left(  z^* (s)+z^* (h(s)) \right)  \, ds\right.\\
&\qquad\qquad \left.- \int_{0}^{t_1}\Psi^{'}(s)(\Psi(t_1)-\Psi(s))^{\alpha-1}\,\left(  z^* (s)+z^* (h(s)) \right)  \, ds\right)\\
&\geq \frac{1}{\Gamma(\rho)} \sum_{0<t_k<t_2 - t_1}   \mathcal{L}_{{\mathcal{J}}_{k}} z^* (t_k^-) + \frac{K\,\mathcal{M}(\Psi(b)-\Psi(0))^{1-\rho}}{(1-\mathcal{L}_f)\Gamma(\alpha)}\left( \int_{0}^{t_2}\Psi^{'}(s)(\Psi(t_2)-\Psi(s))^{\alpha-1}\,ds \right.\\
&\qquad \qquad \left. - \int_{0}^{t_1}\Psi^{'}(s)(\Psi(t_1)-\Psi(s))^{\alpha-1}\,ds\right)\\
& = \frac{1}{\Gamma(\rho)} \sum_{0<t_k<t_2 - t_1}   \mathcal{L}_{{\mathcal{J}}_{k}} z^* (t_k^-)+\frac{K\,\mathcal{M}(\Psi(b)-\Psi(0))^{1-\rho}}{(1-\mathcal{L}_f)\Gamma(\alpha)} \left((\Psi(t_2)-\Psi(0))^{\alpha}- (\Psi(t_1)-\Psi(0))^{\alpha}  \right)\\
&>0. 
\end{align*}
\end{small}
This proves that $z^*$ is increasing operator. Since $h(t)\leq t,\, z^*(h(t)) \leq z^* (t), ~t\in [0,b].$ Therefore \eqref{e24} reduces to
\begin{small}
\begin{align*}
z^* (t)
& \leq \epsilon\, \left( \frac{m }{\Gamma(\rho)} + \,(\Psi(b)-\Psi(0))^{1-\rho}\, E_\alpha ((\Psi(b)-\Psi(0))^\alpha)\right)  + \frac{1}{\Gamma(\rho)} \sum_{0<t_k<t}   \mathcal{L}_{{\mathcal{J}}_{k}} z^* (t_k^-) \nonumber\\
 &\qquad +  \frac{2\,K\,(\Psi(b)-\Psi(0))^{1-\rho}}{(1-\mathcal{L}_f)\Gamma(\alpha)} \int_{0}^{t}\mathcal{A}_\Psi^\alpha(t,s)\,z^* (s)  \, ds,\, t\in [0,b].
\end{align*}
\end{small}
By applying Lemma \ref{JI} to the above inequality with
\begin{small} 
\begin{align*}
& \mathcal{U}(t)=z^* (t),~ \mathcal{V}(t)=\epsilon\, \left( \frac{m }{\Gamma(\rho)} + \,(\Psi(b)-\Psi(0))^{1-\rho}\, E_\alpha ((\Psi(b)-\Psi(0))^\alpha)\right),\\
& \mathbf{g}(t)= \frac{2\,K\,(\Psi(b)-\Psi(0))^{1-\rho}}{(1-\mathcal{L}_f)\Gamma(\alpha)},\,~\beta_k = \frac{ \mathcal{L}_{{\mathcal{J}}_{k}}}{\Gamma(\rho)},
\end{align*}
\end{small}
we obtain
\begin{small}
\begin{align*}
z^* (t) 
&\leq \epsilon\, \left( \frac{m }{\Gamma(\rho)} + \,(\Psi(b)-\Psi(0))^{1-\rho}\, E_\alpha ((\Psi(b)-\Psi(0))^\alpha)\right)\\
&\qquad \times\left[\prod_{i=1}^{k}\left\lbrace 1+\frac{ L_{\mathcal{J}_{i}}}{\Gamma(\rho)} E_\alpha \left(\frac{2\,K\,(\Psi(b)-\Psi(0))^{1-\rho}}{(1-\mathcal{L}_f)\Gamma(\alpha)}\Gamma(\alpha)(\Psi(t_i)-\Psi(0))^\alpha \right)  \right\rbrace  \right]\\
&\qquad\qquad\times E_\alpha\left(\frac{2\,K\,(\Psi(b)-\Psi(0))^{1-\rho}}{(1-\mathcal{L}_f)\Gamma(\alpha)}\Gamma(\alpha)(\Psi(t)-\Psi(0))^\alpha \right)\\
& \leq \epsilon\, \left( \frac{m }{\Gamma(\rho)} + \,(\Psi(b)-\Psi(0))^{1-\rho}\, E_\alpha ((\Psi(b)-\Psi(0))^\alpha)\right) \left\lbrace 1+\frac{ L_{\mathcal{J}}^{max}}{\Gamma(\rho)} E_\alpha \left(\frac{2\,K\,(\Psi(b)-\Psi(0))^{1-\rho+\alpha}}{(1-\mathcal{L}_f)} \right)  \right\rbrace ^ p\\
&\qquad\qquad\times E_\alpha\left(\frac{2\,K\,(\Psi(b)-\Psi(0))^{1-\rho}}{(1-\mathcal{L}_f)}(\Psi(t)-\Psi(0))^\alpha \right),\,~ t\in J.
\end{align*}
\end{small}
Therefore,
\begin{small}
\begin{equation}\label{e244}
z^* (t) \leq \epsilon\, C_{p,E_\alpha}\, E_\alpha\left(\zeta_{f,\Psi}\,(\Psi(t)-\Psi(0))^\alpha \right),~t\in J,
\end{equation}
\end{small}
where 
\begin{small}
\begin{align*}
& L_{\mathcal{J}}^{max}=\max\{L_{\mathcal{J}}^{1}, L_{\mathcal{J}}^{2}, \cdots, L_{\mathcal{J}}^{p}\},~\zeta_{f,\Psi} =  \frac{2\,(\Psi(b)-\Psi(0))^{1-\rho}}{(1-\mathcal{L}_f)}\\
 & C_{p,E_\alpha} = \left( \frac{m }{\Gamma(\rho)} + \,(\Psi(b)-\Psi(0))^{1-\rho}\, E_\alpha ((\Psi(b)-\Psi(0))^\alpha)\right)\\
 &\qquad \qquad\times \left\lbrace 1+\frac{ L_{\mathcal{J}}^{max}}{\Gamma(\rho)} E_\alpha \left(\frac{2\,K\,(\Psi(b)-\Psi(0))^{1-\rho+\alpha}}{(1-\mathcal{L}_f)} \right)  \right\rbrace ^ p.
\end{align*}
\end{small}
Note that for $z(t)=(\Psi(t)-\Psi(0))^{1-\rho}\, |v(t)-u(t)|$ from \eqref{e22} we have $z\leq \mathcal{Q} (z)$, where $\mathcal{Q}$ is an increasing Picard operator. Therefore by Lemma \ref{AGL} we obtain $z\leq z^*$. This fact in combination with \eqref{e244} gives
\begin{small}
\begin{equation}\label{JK}
(\Psi(t)-\Psi(0))^{1-\rho}\, |v(t)-u(t)|\leq \epsilon\, C_{p,E_\alpha}\, E_\alpha\left(\zeta_{f,\Psi}\,(\Psi(t)-\Psi(0))^\alpha \right), \,t\in J.
\end{equation}
\end{small}
Thus, we have proved that the problem \eqref{p1}--\eqref{p2} is UHML stable.
\begin{rem}
Since $E_\alpha (\cdot)$ is increasing, the inequality \eqref{JK} can be written as
\begin{small}
$$
(\Psi(t)-\Psi(0))^{1-\rho}\, |v(t)-u(t)|\leq \epsilon\, C_{p,\,E_\alpha}\, E_\alpha\left(\zeta_{f,\Psi}\,(\Psi(b)-\Psi(0))^\alpha \right),\, \text{for all}\, \, t\in J.
$$
\end{small}
Further, $|u(t)-v(t)|=0,~ t\in [-r,0]$. Therefore, $\|u-v\|_{\mathcal{C}}$.
Thus
\begin{small}
\begin{equation}\label{e25}
\|v-u\|_{{\mathcal{PC}}_{1-\rho;\, \psi}(J,\mathbb{R})} \leq \epsilon \, C_f,
\end{equation}
\end{small}
where
\begin{small}
 $C_f= C_{p,E_\alpha}\, E_\alpha\left(\zeta_{f,\Psi}\,(\Psi(b)-\Psi(0))^\alpha \right).$
 \end{small}
Further, for
\begin{small}
 $t\in [-r,0], |v(t)-u(t)|=0,$ 
 \end{small}
 this implies that 
\begin{small}
\begin{equation}\label{e26}
\|v-u\|_{C[-r,0]}=0.
\end{equation}
\end{small}
Thus, from \eqref{e25} and \eqref{e26} and by definition of $\|.\|_{\mathbb{X}_{\mathcal{C},\,\rho,\,\Psi}}$ we have,
$\|v-u\|_{\mathbb{X}_{\mathcal{C},\,\rho,\,\Psi}}\leq \epsilon \, C_f.$ This proves that the problem \eqref{p1}--\eqref{p2} is Ulam--Hyers stable. Further by defining $\theta (\epsilon)=\epsilon \, C_f,$ we get generalized Ulam--Hyers stablity.
\end{rem}
\section{Examples}
\begin{ex}
Consider the following implicit impulsive $\Psi$--HFDE with delay
\begin{small}
\begin{align}\label{e17}
\begin{cases}
 ^\mathcal{H} \mathbb{D}^{\alpha,\, \beta; \, \Psi}_{0^+}u(t)=\frac{\left( \Psi(t)-\Psi(0)\right) ^{1-\rho}}{50\, e^{(\Psi(t)-\Psi(0))}\,\left( 1+|u(t)|+|u(t-{\frac{1}{2}})|\right) } + \frac{\left| ^\mathcal{H} \mathbb{D}^{\alpha,\, \beta; \, \Psi}_{0^+}u(t)\right| }{15 \left( 1+\left| ^\mathcal{H} \mathbb{D}^{\alpha,\, \beta; \, \Psi}_{0^+}u(t)\right| \right) },~t \in J=(0,1]-\{\frac{1}{3}\},\\
\Delta \mathbb{I}_{0^+}^{1-\rho; \, \Psi}u(\frac{1}{3}^-)= \frac{\left( \Psi(\frac{1}{3})-\Psi(0)\right) ^{1-\rho} \left| u(\frac{1}{3}^-)\right| }{7 \left( 1+ \left| u(\frac{1}{3}^-)\right| \right) },   \\
 \mathbb{I}_{0^+}^{1-\rho; \, \Psi}u(0)= u_0 \in \mathbb{R} , \\
 u(t)=0,\, ~ t\in [-1,0]. 
\end{cases}
\end{align}
\end{small}
\end{ex}
Define 
\begin{small}
$f: (0,1]\times {\mathbb{R}}^3 \to \mathbb{R}$
\end{small}
 by 
\begin{small}
$$
f\left(t, u, v, w \right)= \frac{\left( \Psi(t)-\Psi(0)\right) ^{1-\rho}}{50\, e^{(\Psi(t)-\Psi(0))}\,\left( 1+|u|+|v|\right) } + \frac{\left| w \right| }{15 \left( 1+\left| w \right| \right) }
$$
\end{small}
and 
\begin{small}
$\mathcal{J}_1:\mathbb{R}\to \mathbb{R}$
\end{small}
 by
 \begin{small}
$$
\mathcal{J}_1(u)= \frac{\left( \Psi(\frac{1}{3})-\Psi(0)\right) ^{1-\rho} \left| u \right| }{7 \left( 1+ \left| u \right| \right) }.
$$
\end{small}
Then $f$ satisfies ($H_1$) indeed for $u_i, v_i, w_i \in \mathbb{R}, \text{for}~\, i=1,2$ and for $t\in (0,1]$ we have
\begin{small}
\begin{align*}
&\left|f\left(t, u_1, v_1, w_1 \right)-f\left(t, u_2, v_2, w_2 \right) \right| \\
&\leq \frac{\left( \Psi(t)-\Psi(0)\right) ^{1-\rho}}{50\, e^{(\Psi(t)-\Psi(0))}}\, \left|\frac{1}{ 1+|u_1|+|v_1|} -\frac{1}{ 1+|u_2|+|v_2|} \right| +\frac{1}{15}\, \left|\frac{\left| w_1 \right| }{ 1+\left| w_1 \right|  }-\frac{\left| w_2 \right| }{ 1+\left| w_2 \right|  } \right| \\
& \leq \frac{\left( \Psi(t)-\Psi(0)\right) ^{1-\rho}}{50}\, \left( \left|\,u_1- u_2\,\right| +\left|\, v_1-v_2\,\right| \right) + \frac{1}{15} \left|\,w_1-w_2\, \right|. 
\end{align*}
\end{small}
This implies the $f$ satisfies the hypothesis ($H_1$) with $\mathcal{L}_f=\frac{1}{15}$ and $K=\frac{1}{50}$. Further, for any $u.v\in \mathbb{R}$,
\begin{small}
\begin{align*}
\left|\mathcal{J}_1 (u)-\mathcal{J}_1 (v) \right| & =  \frac{\left( \Psi(\frac{1}{3})-\Psi(0)\right) ^{1-\rho}}{7}\left(\left|\,\frac{\left| u \right| }{ 1+\left|u \right|  }-\frac{\left| v \right| }{ 1+\left| v \right|  }\, \right| \right) \leq \frac{\left( \Psi(\frac{1}{3})-\Psi(0)\right) ^{1-\rho}}{7} |u-v|. 
\end{align*}
\end{small} 
 This shows that $\mathcal{J}_1 $ satisfy the assumptions ($H_2$) with $L_{{\mathcal{J}}_1}=\frac{1}{7}$. Thus in the view of Theorem \ref{exi} the problem \eqref{e17} has unique solution if the condition,
 \begin{small}
\begin{equation}\label{e188}
L=\left( \frac{1}{7\,\Gamma(\rho)} + \frac{3\,(\Psi(1)-\Psi(0))^{1-\rho+\alpha}}{70 \,\Gamma(\alpha+1)} \right) < 1 ,
\end{equation}
\end{small}
is satisfied. 
In addition for every solution $v \in \mathbb{X}_{\mathcal{C},\,\rho,\,\Psi}$ of the inequality
\begin{small}
\begin{align}
\begin{cases}
\left| ^\mathcal{H} \mathbb{D}^{\alpha,\beta;\,\Psi}_{0^+}  v(t)-\frac{\left( \Psi(t)-\Psi(0)\right) ^{1-\rho}}{50\, e^{(\Psi(t)-\Psi(0))}\,\left( 1+|u|+|v|\right) } - \frac{\left| w \right| }{15 \left( 1+\left| w \right| \right) }\right| \leq \epsilon\, E_\alpha\left((\Psi(t)-\Psi(0))^\alpha \right),~ t \in (0,1],\\
 |\Delta \mathbb{I}_{0^+}^{1-\rho; \, \Psi}v(t_k)- \frac{\left( \Psi(\frac{1}{3})-\Psi(0)\right) ^{1-\rho} \left| u \right| }{7 \left( 1+ \left| u \right| \right) }| \leq \epsilon,  
\end{cases}
\end{align}
\end{small}
there exists unique solution of the problem \eqref{e17} such that
\begin{small}
$$
(\Psi(t)-\Psi(0))^{1-\rho}\, |v(t)-u(t)|\leq \epsilon\, C_{p,E_\alpha}\, E_\alpha\left(\zeta_{f,\Psi}\,(\Psi(t)-\Psi(0))^\alpha \right), \,t\in (0,1],
$$
\end{small}
where
\begin{small}
\begin{align*}
C_{p,E_\alpha} &= \left( \frac{1 }{\Gamma(\rho)} + \,(\Psi(1)-\Psi(0))^{1-\rho}\, E_\alpha ((\Psi(1)-\Psi(0))^\alpha)\right)\, \left\lbrace 1+\frac{ 1}{7\,\Gamma(\rho)} E_\alpha \left(\frac{3\,(\Psi(1)-\Psi(0))^{1-\rho+\alpha}}{70} \right)  \right\rbrace,\\
\zeta_{f,\Psi}& =  \frac{15\,(\Psi(1)-\Psi(0))^{1-\rho}}{7}.
\end{align*}
\end{small}

In particular, take $\alpha={\frac{1}{2}}, \beta=1$, then $\rho=1$.  Let $\Psi(t)=t$ and $h(t)=t-{\frac{1}{2}},\, t\in [0,1]$.  Then the problem \eqref{e17} reduces to the following implicit impulsive Caputo FDE
\begin{small}
\begin{align}\label{e18}
\begin{cases}
 ^C \mathbb{D}^{\frac{1}{2}}_{0^+}u(t)=\frac{1}{50\, e^{t}\,\left( 1+|u(t)|+|u(t-{\frac{1}{2}})|\right) } + \frac{\left| ^C \mathbb{D}^{\frac{1}{2}}_{0^+}u(t)\right| }{15 \left( 1+\left| ^C \mathbb{D}^{\frac{1}{2}}_{0^+}u(t)\right| \right) },~t \in J=(0,1]-\{\frac{1}{3}\},\\
\Delta u(\frac{1}{3}^-)= \frac{\left| u(\frac{1}{3}^-)\right| }{7 \left( 1+ \left| u(\frac{1}{3}^-)\right| \right) },   \\
  u(t)=0,~t\in [-1,0]. 
\end{cases}
\end{align} 
\end{small} 
In this case, from \eqref{e188}, $\mathcal{L}\approx 0.1912<1$, and hence \eqref{e18} has unique solution and the corresponding problem is UHML stable.

Further, for $\Psi(t)=t$,  $\alpha= \frac{1}{3}$, $\beta=0$, we have $\rho=\frac{1}{3}$ and the problem \eqref{e17} reduces to implicit impulsive Riemann-Liouville FDE with delay. In this case $\mathcal{L}\approx 0.1013<1$. Hence the implicit impulsive Riemann-Liouville FDE has unique solution and the corresponding problem is UHML stable.
\end{proof}

\section*{Acknowledgment}
 The first author thanks Shivaji University, Kolhapur for the grants through research strengthening scheme. The second author  acknowledges the Science and Engineering Research Board (SERB), New Delhi, India for the Research Grant (Ref: File no. EEQ/2018/000407).


\begin{thebibliography}{99}

\bibitem{Fečkan}
J. Wang,  M. Fečkan,  Y. Zhou,  A survey on impulsive fractional differential equations,
 Fractional Calculus and Applied Analysis 19 (4) (2016) 806--831.
 
 \bibitem{Wang1} 
   J. Wang, Y. Zhou, Z. Lin,
    On a new class of impulsive fractional differential equations,
    Applied Mathematics and Computation 242 (2014) 649--657.
    
  \bibitem{Wang2} J.Wang, M. Fe$\breve{c}$kan, Y. Zhou, Ulam’s type stability of impulsive ordinary differential equations, Journal of Mathematical Analysis and
  Applications 395 (2012) 258--264.
 
\bibitem{Wang22} 
   J. Wang, Y. Zhou, M. Fe$\breve{c}$kan,
  On recent developments in the theory of boundary value problems for   impulsive fractional differential equations,
  Computers and Mathematics with Applications 64 (2012) 3008--3020.  
 
 \bibitem{fec}
  M. Feckan, Y. Zhou, J. Wang,
  On the concept and existence of solution for impulsive fractional differential equations, 
  Commun Nonlinear Sci Numer Simulat 17 (7) (2012) 3050--3060.
 
 \bibitem{Benchohra3}
     M. Benchohra, B. Slimani,
     Existence and uniqueness of solutions to impulsive fractional differential equations,
     Electronic Journal of Differential Equations 2009 (10) (2009) 1--11. 
      
 \bibitem{Mophou} 
 G. M. Mophou,
 Existence and uniqueness of mild solutions to impulsive 
 fractional differential equations, Nonlinear Analysis 72 (2010) 1604--1615. 
 
  \bibitem{JNieto} Juan J. Nieto , Abelghani Ouahab , Venktesh Venktesh,
   Implicit fractional differential equations via Liouville -Caputo derivative,
    Mathematics 3 (2015) 398-411.
 
  \bibitem{Benchora2} M. Benchohra, J. Lazreg,
   On stability for nonlinear implicit fractional differential equations,
   Le Matematiche  Vol. LXX (2015)  Fasc. II,  49-61.
   
   \bibitem{Benchora3} M. Benchohra, M. Said Souid,
   Integrable solutions for implicit fractional order differential equations,
    Nonlinear Analysis 69 (2008) 2677-2682.
    
  \bibitem{Kucche} K. D. Kucche, J. J. Nieto, V. Venktesh,
   Theory of nonlinear implicit fractional differential equations,
   Differ. Equ. Dyn. Syst. DOI 10.1007/s12591-016-0297-7.
   
   \bibitem{Shah}
    K. Shah, A. Ali, S. Bushnaq, 
   Hyers-Ulam stability analysis to implicit Cauchy problem of fractional differential equations with impulsive conditions,
   Math. Meth. Appl. Sci. 2018 (41) (2018) 8329-–8343.   
    
  \bibitem{Eghbali}
  N. Eghbali, V. Kalvandi, J. M. Rassias,
   A fixed point approach to the Mittag-Leffler-Hyers-Ulam stability of a fractional integral equation, 
   Open Math. 14 (2016) 237–-246.
   
  \bibitem{Zhang} 
  J. R. Wang, Y. Zhang,
   Ulam-Hyers-Mittag-Leffler stability of fractional-order delay differential equations, 
   Optimization 63 (2014) 1181–-1190. 
   
 
 \bibitem{Niazi}
 A. U. K. Niazi, J. Wei, M. U. Rehman , P. Denghao,
  Ulam-Hyers-Mittag-Leffler stability for nonlinear fractional neutral differential equations,
 Sbornik: Mathematics 209 (9) (2018) 1337–-1350.  
   
   
   
 \bibitem{Liu3}
 K. Liu, J. Wang, D. O’Regan,
 Ulam–-Hyers-–Mittag--Leffler stability for $\Psi$-Hilfer fractional-order delay differential equations,
  Advances in Difference Equations (2019) 2019:50.   
  
  \bibitem{KDK}
   J. Vanterler da C. Sousa,  K. D. Kucche ,  E. Capelas de Oliveira,
   Stability of ψ-Hilfer impulsive fractional differential equations,
    Appl. Math. Lett. 88 (2019) 73--80.
     
 \bibitem{Kilbas}
  A. A. Kilbas, H. M. Srivastava, J. J. Trujillo,
  Theory and applications of fractional differential equations,
   Elsevier. Science, B.V. Amsterdam (2006).   
    
 \bibitem{Sousa1} 
 J. Vanterler da C. Sousa, E. Capelas de Oliveira,
  On the  $\psi$--Hilfer fractional derivative,
   Commun. Nonlinear Sci. Numer. Simulat. 60 (2018) 72–-91.
   
   \bibitem{Bai}Z. Bai, X. Dong, C. Yin, Existence results for impulsive nonlinear fractional differential equation with mixed boundary conditions, Boundary Value Problems  (2016) 2016:63.   
 
 \bibitem{zhou} Zhou Y., Basic theory of fractional differential equations. World scientific (2014).
 
\bibitem{rus}
I.A.Rus, Growall lemmas: ten open problems,
 Sci. Math. Jpn. 70 (2009) 221--228.
 
 \bibitem{Jose} 
   J. Vanterler da C. Sousa, D.S. Oliveira, E. Capelas de Oliveira,
   A note on the mild solutions of Hilfer impulsive fractional differential equations, arXiv:1811.09256 (2018).
 
 \bibitem{JK}
 K. D. Kucche, J. P. Kharade,  J. Vanterler da C. Sousa,
 On the Nonlinear Impulsive $\Psi$–Hilfer Fractional Differential Equations, 	arXiv:1901.01814 (2019). 
 
\bibitem{Sousa2} J. Vanterler da C. Sousa, E. Capelas Oliveira,
 A Gronwall inequality and the Cauchy-type problem by means of $\psi$-Hilfer operator, Diff. Equ. \&  Appl. 11 (1) (2019) 87--106. 
 
  
 	
 
 
\end{thebibliography}
\end{document}